\documentclass{article}
\usepackage{lmodern}			
\usepackage[usenames,dvipsnames]{pstricks}
\usepackage{stackrel}
\usepackage{float}
\usepackage{graphicx}
\usepackage{epstopdf}
\usepackage{enumerate}
\usepackage{amsthm}
\usepackage{multicol}
\usepackage{amssymb,indentfirst} 
\usepackage{amssymb}
\usepackage[centertags]{amsmath} 
\usepackage[utf8]{inputenc}		
\newenvironment{prova}[1][\textbf{Proof}]{\noindent\textit{#1. }}{\\ \qed}
\newtheorem{teo}{Theorem}[section]
\newtheorem{cor}[teo]{Corollary}

\newtheorem{defi}[teo]{Definition}

\newtheorem{obs}[teo]{Remark}

\usepackage{array}
\newcommand{\R}{\mathbb{R}}

\newcommand{\h}{\mathbb{H}}

\newcommand{\sech}{\operatorname{sech}}

\newcommand{\pati}{\textcolor{red}}
\newcommand{\mat}{\textcolor{blue}}
\title{Invariant solutions for the asymptotic Plateau problem in $\h^3$}
\author{Matheus Pimentel Gomes
\\
Patrícia Kruse Klaser \\
Jaime Bruck Ripoll \\
Leonardo Prange Bonorino \\}
\date{}

\begin{document}

\maketitle

\begin{abstract}
In this paper, we present solutions to the asymptotic Plateau problem in the hyperbolic space $\mathbb{H}^3$. In this context, we exhibit solutions for curves that are invariant under the action of a one-parameter subgroup of isometries of $\mathbb{H}^3$. To achieve this, we prove the existence of foliations of $\mathbb{H}^3$ by minimal surfaces that are properly embedded, complete, and invariant under these subgroups, which are then used to solve the problem.
\end{abstract}

\section{Introduction}

In this paper we study the asymptotic Plateau problem (APP) in the hyperbolic space $\h^3$, the complete simply connected Riemannian manifold with constant sectional curvature  -1. The main question of this problem is about the existence of a minimal surface with asymptotic boundary $\Gamma$, where $\Gamma$ is a given curve in the asymptotic boundary of $\h^3$ ($\partial_{\infty}\h^3$). In fact, this problem has been widely studied in $\h^3$ and it has been solved for a large class of curves. We are also interested in describing the solutions to this problem in some particular cases.

In \cite{Anderson}, Anderson proved for each immersed closed curve $\Gamma$ in the asymptotic boundary of $\h^3$ there exists a complete, absolutely area-minimazing locally integral 1-current $\Sigma$ in $\h^3$ such that the asymptotic boundary of $\Sigma$ is $\Gamma$. Guan and Spruck in \cite{Spruck} considered the asymptotic Dirichlet problem (ADP) for hyperbolic Killing graphs and exhibited a family of curves $\Gamma \subset \partial_{\infty}\h^3$ (precisely the curves that correspond to a graph of a continuous function defined on $\partial_{\infty}\h^2$) that can be minimally filled, that is, there is $\Sigma$ a complete, embedded simply connected minimal surface with asymptotc boundary $\Gamma.$ They also proved this solution is unique and it is a radial graph. Furthermore, the forth author and Telichevesky in \cite{Miriam} generalized Guan and Spruck's result to parabolic Killing graphs.

In this paper we are interested in the asymptotic Plateau problem for invariant curves. A Killing field in $\h^3$ is of one of the following types: elliptic, parabolic, hyperbolic or helicoidal. The elliptic Killing fields have the flux associated to the rotation around a geodesic $\gamma$ and fix every point in $\gamma.$ A parabolic Killing field has the flux associated to the rotantion around a point $p$ at infinity and every orbit is a horocicle containg $p.$ A hyperbolic Killing field has the flux associated to the translation along a geodesic $\gamma$ and every orbit of its flux is a curve equidistant to $\gamma.$ Finally, the helicoidal Killing fields are obtained as the sum of a elliptic and a hyperbolic field with the same invariant geodesic $\gamma.$ In any case the action of the flux associated to the field extends naturally to the asymptotic boundary of $\h^3.$ A curve $C \subset \partial_{\infty}\h^3$ is called invariant by the field $X$ if the flux of $X$ in $C$ does not change it.

Given $C \subset \partial_{\infty}\h^3$ an invariant curve, we are interested in the question of whether $C$ bounds an invariant symply connected minimal surface. The simplest situation occurs when one considers an elliptic Killing field $X:$ Let $\gamma$ be the geodesic fixed by $X$. Through every point $\gamma(t)$ of $\gamma,$ there is a totally geodesic surface $\h^2(t)$ orthogonal to $\gamma.$ The asymptotic boundaries of these $\h^2(t)$ form a family of circles in $\partial_{\infty}\h^3$ which are invariant by the flux of $X$ and are minimaly filiable. These surfaces are well known in the literature and also correspond to the solution of the ADP considered in \cite{Spruck} and \cite{Miriam} for constant boundary data.

An invariant curve $C$ at $\partial_{\infty}\h^3,$ associated to the other three types of Killing fields, may not be a Killing graph associated to a parabolic or hyperbolic Killing field. For these cases, the knowledge about the APP is restricted to Andersons result. Nevertheless invariant minimal surfaces were studied in \cite{Manfredo} and they are the surfaces we use to prove the following result

\begin{teo}\label{exis}
    Let $G$ be 1-parameter subgroup of $Iso(\h^3)$ and let $\Gamma$ be a $G-$ invariant connected closed curve at $\partial_\infty\h^3,$ that bounds a connected region of $\partial_\infty\h^3.$ Then there exists a complete properly embedded simply connected $G-$invariant minimal surface $\Sigma$ with $\partial_\infty\Sigma=\Gamma$.
\end{teo}

The technique used to prove this theorem is based on foliations of the manifold. To this end, we use the fact that each of these Killing fields (except the elliptic one) induces a one-parameter group of isometries acting on $\mathbb{H}^3$. Under suitable assumptions, we can then associate a quotient manifold $\mathbb{H}^3 / G$ and a Riemannian submersion from $\mathbb{H}^3$ onto $\mathbb{H}^3 / G$, where $G$ is the group of isometries associated with the Killing field. In this way, we show that it is possible to obtain foliations of $\mathbb{H}^3$ from foliations of $\mathbb{H}^3 / G$. Moreover, such foliations can be extended to the asymptotic boundary of $\mathbb{H}^3$, as we prove in the following theorem:

\begin{teo}\label{main}
    Let $G$ be a $1$-parameter subgroup of the isometries of $\h^3$. Then there are $G$-invariant connected open subsets $\{\Omega_s\}_{s\in\R}$ of $\overline{\h^3}$ satisfying:
    \begin{itemize}    
        \item[(i)] $\Omega_s\subset\Omega_t$, if $s\leq t$, and $\displaystyle \h^3 = \bigcup_{s\in\R}\Omega_s;$ 
        \item[(ii)] for each $s\in\R,$ $\Sigma_s=\partial\Omega_s$ is a complete properly embedded minimal $G$-invariant surface and $\{\Sigma_s\}_{s\in\R}$ foliate $\h^3$;
            \item[(iii)] If $G$ is elliptic, then $\displaystyle\{\overline{\Sigma}_s\}_{s\in\R}$ foliates $\overline{\h}^3\backslash \{p,q\}$, where $p$ and $q$ are the $G$-invariant points. Otherwise, $\displaystyle\{\overline{\Sigma}_s\}_{s\in\R}$ foliates $\overline{\h}^3\backslash C$, where $C$ is a $G$-invariant curve in $\partial_\infty\h^3$.
    \end{itemize}    
\end{teo}

The minimal surfaces presented in this work have already been described by other authors. The parabolic and hyperbolic surfaces were studied in \cite{Manfredo}, while the helicoidal ones were described in \cite{heli}. However, the corresponding foliations were not established in the parabolic and hyperbolic cases in the aforementioned works. In \cite{heli}, the fourth author proved that existence of foliations of $\h^3$ by helicoidal minimal surfaces for the case where the helicoidal Killing field, which is a combination of an elliptic and a hyperbolic field has the angle with the hyperbolic field smaller than the angle with the elliptic one.  In the present work, we generalize this existence result for any helicoidal field. Moreover, we present these families of minimal surfaces in a way that differs from those previously studied.

Finally, we organize this work as follows: in Section 2, we present results related to foliations, and in Section 3, we describe foliations of hyperbolic space by parabolic, hyperbolic, and helicoidal minimal surfaces, as well as how one of these foliations extends to the asymptotic boundary in order to solve the asymptotic Plateau problem.



\section{The Foliation Theorem}

In this section we prove the Foliation Theorem. First, let us set some basic notation and present a minimal version of Proposition $2.3$ from \cite{nelli}.

Let $(M,\overline{g})$ be a complete three-dimensional Riemannian manifold and $X$ a Killing field. Denote by $G$ the one-parameter subgroup of isometries associated with $X$ in $M$ and suppose $G$ acts freely and properly on $M$; hence one can define the quotient manifold. Let $M/G$ be the quotient manifold and equip $M/G$ with the metric $g$ such that the natural projection $\pi:(M,\overline{g})\rightarrow (M/G,g)$ is a Riemannian submersion. We denote $\pi^{-1}(p)$ by $G(p):=\{h(p)\in M \mid h\in G\}$, that is, $G(p)$ is the orbit of $p$ in $M$.

Let $v:M\rightarrow\R$ be the function defined by $v(p)=\overline{g}(X,X)(p)$. Since $X$ is a Killing field on $M$, $v$ is constant along $G(p)$ and consequently $v$ is well defined on $M/G$.

\begin{defi}\label{def-inv}
    Let $\Sigma$ be a surface in $M$ and $G \subset Iso(M)$ a group. We say that $\Sigma$ is $G$-invariant if $G(\Sigma)\subset\Sigma$.
\end{defi}

\begin{teo}\label{red}

  Under the above conditions, let
  $G\subset Iso(M)$ be the one-parameter subgroup associated with the Killing field $X$ and let $\Sigma\subset M$ be a surface invariant under $G$. Let $\alpha$ be a parametrization of the curve $\pi(\Sigma)\subset (M/G,g^*)$ such that $g^*(\alpha',\alpha')\equiv 1$, where $g^*=vg$. Then $\Sigma$ is a minimal surface if and only if $\alpha$ is a geodesic in $(M/G,g^*)$.
\end{teo}

Theorem \ref{red} is a minimal version of Proposition $2.3$ from \cite{nelli}. Now we can prove the Foliation Theorem.

\begin{teo}\label{teo-fol}

  Under the conditions of Theorem \ref{red} and suppose $M$ is a orientable surface, $M$ admits a foliation $\mathcal{F}$ by $G$-invariant minimal surfaces if and only if there exists a complete embedded curve $\Gamma$ in $M/G$ such that the map $\exp^{\perp}:(T\Gamma)^\perp\rightarrow M/G$, given by $\exp^\perp(q,v)=\exp(q,v)$, where $q\in\Gamma$ and $v\in T_q\Gamma^\perp$, is a diffeomorphism. Furthermore $\Sigma\in\mathcal{F}$ if and only if there is a geodesic $\gamma$ of $(M/G,g^*)$ orthogonal to $\Gamma$ such that $\Sigma=\pi^{-1}(\gamma)$.
  
\end{teo}

\begin{prova}
    Suppose there exists a curve $\Gamma$ in $M/G$ such that $\exp^\perp:(T\Gamma)^\perp \rightarrow M/G$ is a diffeomorphism. Given $p\in\Gamma$, consider $\alpha:(-\epsilon,\epsilon)\rightarrow M/G$ to be a parametrization of $\Gamma$ in a neighborhood of $p$ such that $\alpha(0)=p$. Define $F:(-\epsilon,\epsilon)\times\mathbb{R}\rightarrow M/G$ by $F(s,t)=\exp^\perp\big(\alpha(s),tv\big)$, with $|v|=1$. Since we can perform this process for each $p\in\Gamma$, we obtain a foliation $\mathcal{F}^*$ of $M/G$ by geodesics orthogonal to $\Gamma$. Now define the set
\[
\mathcal{F}:=\{\Sigma\subset M \mid \Sigma=\pi^{-1}(\gamma),\ \gamma\in\mathcal{F}^*\}.
\]
Since $\mathcal{F}^*$ is a foliation of $M/G$, $\mathcal{F}$ is a foliation of $M$ by $G$-invariant surfaces. Furthermore, by the Reduction Theorem, each leaf of $\mathcal{F}$ is a minimal surface.

Now suppose there exists a foliation $\mathcal{F}$ of $M$ by minimal surfaces that are $G$-invariant. Set
\[
\mathcal{F}^*:=\{\gamma\subset M/G \mid \gamma=\pi(\Sigma),\ \Sigma\in\mathcal{F}\}.
\]
Thus $\mathcal{F}^*$ is a foliation of $M/G$. Furthermore, since each leaf of $\mathcal{F}$ is a $G$-invariant minimal surface, by the Reduction Theorem each leaf of $\mathcal{F}^*$ is a geodesic in $(M/G,g^*)$. Observe that, since $M$ is orientable, the quotient $M/G$ is also orientable. Since $\mathcal{F}^*$ is a foliation of $M/G$ by geodesics, we can define a continuous unit vector field $X$ on $M/G$ orthogonal to the leaves of the foliation, in such a way that $X$ together with the velocity vector of each geodesic forms a positively oriented basis. Therefore, $X$ is globally defined. Consequently, we can define the curve $\Gamma$ as an integral curve of the vector field $X$, and since $X$ is unitary, $\Gamma$ is complete.
\end{prova}

We can apply the Foliation Theorem to a very large class of manifolds — Hadamard manifolds, as we will see as a corollary of this theorem in the next result.

\begin{cor}\label{coro}
   In the context of the previous results, if $(M/G,g^*)$ is a Hadamard manifold, then $M$ admits a foliation by complete properly embedded $G$-invariant minimal surfaces.
\end{cor}

\begin{prova}
   Let $\gamma$ be a geodesic in $(M/G,g^*)$. Set $\exp^\perp:(T\gamma)^\perp\rightarrow (M/G,g^*)$. Since $(M/G,g^*)$ is a Hadamard manifold, $\gamma$ is a complete embedded curve in $(M/G,g^*)$ and $\exp^\perp$ is a diffeomorphism. By the Foliation Theorem, there is a foliation $\mathcal{F}$ of $M$ by $G$-invariant minimal surfaces. 
\end{prova}

\begin{obs}
    When the sectional curvature of a Hadamard manifold is strictly negative, the foliation extends to the asymptotic boundary.
\end{obs}

\section{Foliations of the hyperbolic space $\h^3$}\label{sec}

In this section, we use the results of the previous section together with some integration of ODEs in order to describe the complete simply connected minimal surfaces of $\h^3$ that are invariant under the action of a one-parameter subgroup of $Iso(\h^3)$. 

Let
\begin{equation} \label{eq-H3ret}
\h^3:=\{(x,y,z)\in\R^3 \mid z>0\}, \qquad ds^2=\frac{dx^2+dy^2+dz^2}{z^2}
\end{equation}
be the upper half-space model of $\h^3$.
We denote by $\overline{g}$ the metric associated with this inner product.

The asymptotic boundary of $\h^3$ can be identified with $\{z=0\} \cup \{\infty\}$.

A Killing field in $\h^3$ is of one of the following four types: elliptic, parabolic, hyperbolic and helicoidal. The classification is based on the kind of object obtained by applying the flow of the Killing field at a generic point of the space (a circle, a horocircle, a curve equidistant from a geodesic or a helix). This is an intrinsic property of $\h^3$, but we choose to use a model in order to simplify the computations and figures. A proof of this classification can be found in Section 3 of \cite{Jaime}, by Fornari and Ripoll. The action of an isometry associated with a Killing field extends naturally to the asymptotic boundary of $\h^3$.


\begin{obs}\label{obss}
    Since the action associated with an elliptic Killing field is not free, we cannot apply the techniques presented in this section. However, a hyperbolic plane is an elliptic-invariant surface and a family of hyperbolic planes foliates $\overline{\h^3}$. Thus we can use this family of hyperbolic planes to prove Theorem \ref{main} when $X$ is a rotational Killing field.

    \begin{figure}[H]
    \centering
    \caption{Family of hyperbolic planes $\Sigma_s$}
    \vspace{0.5cm}\includegraphics[scale=0.35]{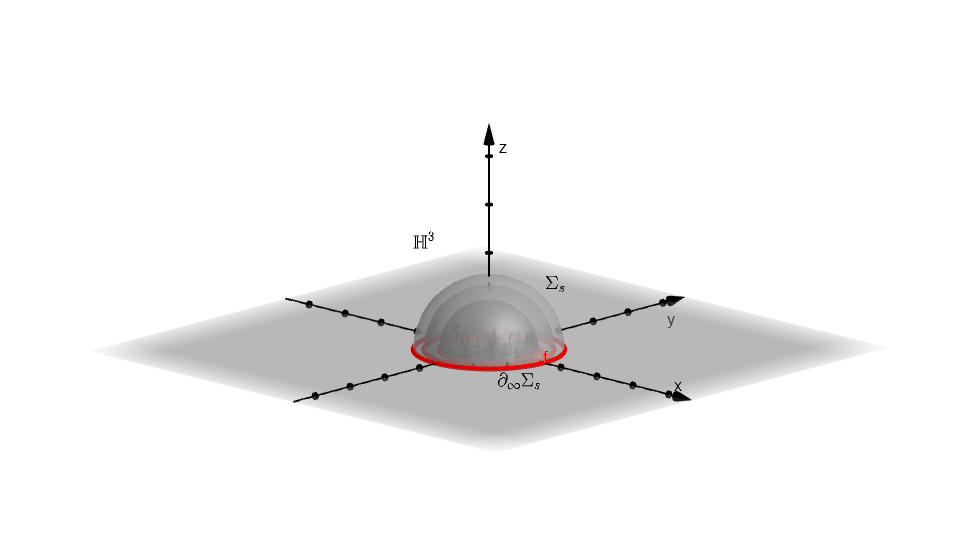}
\end{figure}
\end{obs}

\subsection{Parabolic Killing fields}

In this subsection we present a foliation of $\h^3$ by parabolic minimal surfaces.

\begin{defi}
    A Killing field in $\h^3$ is called parabolic if its orbits are horocircles.
\end{defi}

An invariant surface (see Definition \ref{def-inv}) is called parabolic if it is invariant under the subgroup of isometries associated with a parabolic Killing field.

Given a parabolic Killing field $X$, we may assume, after applying an isometry of $\h^3$, that $X:=a\frac{\partial}{\partial y}$ for some $a\in\R$, $a\neq 0$. Since we are interested in invariant surfaces, we may assume $a=1$. Let  
$\h^2=\{(x,y,z)\in \h^3 \mid y=0\}\subset\h^3$, which is a totally geodesic surface orthogonal to $X$. We can use this setting to compute the metric $g^*$ from Theorem \ref{red}.

Let $G$ be the $1$-parameter subgroup of $Iso(\h^3)$ associated with $X$ and define the projection
\begin{eqnarray*}
    \pi: \h^3 & \rightarrow & \h^2 \\
    (x,y,z) & \mapsto & (x,0,z).
\end{eqnarray*}
Since $X$ is orthogonal to $\h^2$, the metric induced by $\h^3$ on $\h^2$ makes $\pi$ a Riemannian submersion. Therefore the metric $g^*$ on $\h^3/G$ is 
$$\displaystyle ds^2=\frac{dx^2+dz^2}{z^4},$$
since $\displaystyle |X|^2=\frac{1}{z^2}$. Hence the sectional curvature $K$ of $(\h^3/G,g^*)$ is 
$$\displaystyle K(z)=-\frac{2}{z^2}.$$ 
Therefore $(\h^3/G,g^*)$ is a Hadamard manifold and the next result is an application of Corollary \ref{coro}.

\begin{teo}\label{sup.para1}
    For each parabolic Killing field in $\h^3$ there exists a foliation of $\h^3$ by complete properly embedded simply connected parabolic minimal surfaces. 
\end{teo}

In what follows we present another approach to explicitly find a family of parabolic minimal surfaces that foliates $\mathbb{H}^3$. To do this, we look for a generating curve $\gamma$ in the totally geodesic $\h^2=\{y=0\}$ and apply, at each of its points, the flow of the Killing field $X$.

Given a curve $\gamma(t)=(x(t),0,z(t))$, $t \in \R$, in $\mathbb{H}^2$ with $|\gamma'| \equiv 1$, let 
$\Sigma := \{(x(t),s,z(t)) \in \mathbb{H}^3 \mid s, t \in \R\}$ 
be the parabolic surface associated with $X$. Then the mean curvature $H$ of $\Sigma$ is  
\begin{equation}\label{eq.sup.inv.}
-2H = \langle \nabla_{\gamma'}\gamma',\eta \rangle + \Big\langle \nabla_{\frac{X}{|X|}}\frac{X}{|X|},\eta \Big\rangle,
\end{equation}
where $\eta$ is orthogonal to $\Sigma$ and $|\eta| \equiv 1$. Therefore $\Sigma$ is minimal if and only if its mean curvature $H$ vanishes. We can compute equation (\ref{eq.sup.inv.}) in coordinates to obtain

\begin{equation}\label{sis.parab}
\begin{cases}
(x''z' - z''x')z - 2x'z^2 = 0 \\ 
x'^2 + z'^2 = z^2.
\end{cases}
\end{equation}

For $k>0$, we consider the initial condition $\gamma_k(0)=(0,0,k),$ $\gamma_k'(0)=(z^2,0,0)$ and obtain the solutions $\gamma_k(t)=(x_k(t),0,z_k(t))$, $t\in \R$

\begin{equation}
\left\{\begin{array}{lll}
    x_k(t) & = & \frac{1}{\sqrt{k}}\displaystyle{\int_0^t\sech^{\frac{3}{2}}(2v)\,dv}, \\
     \\
    z_k(t) & = & \frac{1}{\sqrt{k}}\sqrt{\sech(2t)}.
\end{array}\right.
\end{equation}

A complete version of this approach can be found in \cite{eu}.

Hence, for each $k>0$, there is a parabolic simply connected embedded minimal surface
$$
\Sigma_k=\{(x_k(t),y,z_k(t)) \mid t\in \R, y\in \R\}.
$$

\begin{obs}
The choice $x_k(0)=0$ is arbitrary since adding a constant to the function $x_k$ keeps it a solution. 
\end{obs}

\begin{obs}\label{rem-doCDac}
    Observe that 
    $$(x_k(t),z_k(t))=\frac{1}{\sqrt{k}}(x_1(t),z_1(t))$$
    and recall that 
    $$\Phi(x,y,z)=\frac{1}{\sqrt{k}}(x,y,z)$$
    is an isometry of $\mathbb{H}^3$ in the half-space model. Therefore the parabolic surfaces $\Sigma_k$ presented above are all congruent to each other by an isometry of the ambient space. These are precisely the minimal surfaces presented by do Carmo and Dajczer in Theorem 3.14 of \cite{Manfredo}.
\end{obs}



The generating curves $\gamma_k$ can also be parametrized with $x$ as a bigraph of a function of $z$ if one observes that
\begin{equation}\label{eq-x1}
x_1(z)=\pm \int_z^1 \frac{z^2}{\sqrt{1-z^4}}\,dz.
\end{equation}

Since Theorem 3.14 of \cite{Manfredo} states that, up to isometries, there is a unique complete parabolic minimal surface in $\h^3$, we have just proved the following result.

\begin{teo}
If $\Sigma$ is a complete parabolic minimal surface in $\h^3$, then up to an isometry of the ambient space it is congruent to
$$
\Sigma_1=\left\{\left(\pm \int_z^1 \frac{z^2}{\sqrt{1-z^4}}\,dz,y,z\right)\mid z\in (0,1],\; y\in \mathbb{R}\right\}.
$$
$\Sigma_1$ is a simply connected surface whose asymptotic boundary consists of two tangent circles at infinity given by
$$
\{(-L,y,0)\mid y\in \R\} \cup \{(L,y,0)\mid y\in \R\}
$$
for
\begin{equation}\label{eq-Lpar}
L=\int_0^1 \frac{z^2}{\sqrt{1-z^4}}\,dz.
\end{equation}
\end{teo}

Since we have the parametrization, we can use the software GeoGebra to plot the surface:

\begin{figure}[H]
    \centering
    \caption{Parabolic minimal surface $\Sigma$}
    \vspace{0.5cm}\includegraphics[scale=0.45]{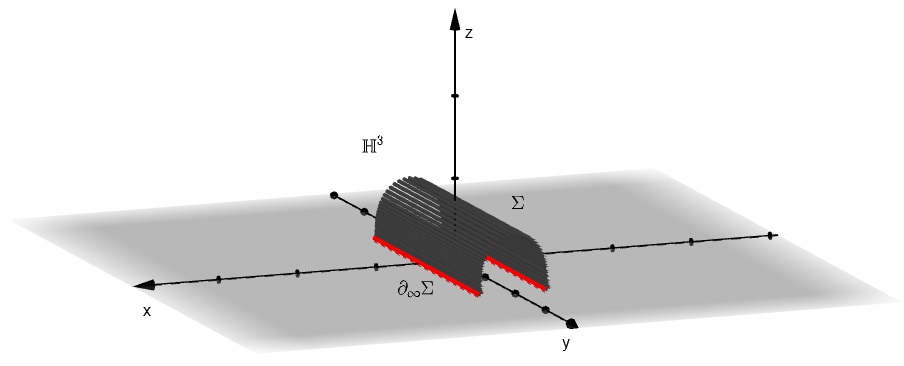}
\end{figure}

Moreover, in order to obtain the curves $\gamma_k$, we followed the proof of the Reduction Theorem and found the geodesics that foliate the quotient space appearing in Theorems \ref{red} and \ref{teo-fol}. Hence the family $\{\Sigma_k\}_{k>0}$ foliates $\h^3$ by complete parabolic embedded simply connected minimal surfaces. Actually, as one can see from their explicit formula, they foliate $\overline{\h^3}\backslash C$, where 
$$
C=\{(0,y,0)\mid y\in \R\}\cup\{\infty\}.
$$

One can also look at these surfaces through the perspective of the asymptotic Plateau problem: the action of the isometry group associated with $X$ extends naturally to the asymptotic boundary of $\h^3$. Thus we have the following result.

\begin{teo}\label{teo-exis-par}
Let $X$ be a parabolic Killing field in $\h^3$ and let $p\in \partial_{\infty}\h^3$ be the unique singularity of $X$. Let $q_1$ and $q_2$ be two points in $\partial_{\infty}\h^3\backslash\{p\}$ with distinct orbits $C_1$ and $C_2$ associated with $X$. Then there exists a unique complete parabolic minimal surface $\Sigma$ with asymptotic boundary $C_1 \cup C_2$.
\end{teo}

\begin{prova}

By fixing the upper half-space model of $\h^3$, one may assume that $X=\frac{\partial}{\partial y}$ and that $C_1$ and $C_2$ are symmetric with respect to the plane $\{x=0\}$. Thus 
$$
C_1=\{(\ell,y,0)\mid y\in \mathbb{R}\}, \qquad 
C_2=\{(-\ell,y,0)\mid y\in \mathbb{R}\}.
$$

Let $L$ be half of the ``width'' of $\Sigma_1$ defined by \eqref{eq-Lpar}. Now consider $k>0$ such that
$$
\frac{1}{\sqrt{k}}L=\ell.
$$
In this case, the asymptotic boundary of $\Sigma_k$ is $C_1 \cup C_2$.
\end{prova}

\begin{cor}\label{cor-11par}
    Theorem \ref{exis} holds if $G$ is a parabolic subgroup of Iso($\h^3$).
\end{cor}

\begin{prova}
    If $\Gamma$ is an invariant closed curve as in Theorem \ref{exis}, then it is either an orbit $G(q)$ for some $q\in \partial_{\infty}\h^3$ or the union of two distinct orbits $G(q_1) \cup G(q_2).$ This is true since it must contain the orbit of each of its elements, and if it contains more than two orbits, it will not bound a connected region in $\partial_\infty\h^3.$

    In the first case, $G(q)$ is the boundary of a totally geodesic $\h^2$ obtained applying the flow of $X$ at the geodesic $\gamma \subset \h^2$ that starts at $q$ and ends at $p,$ the singularity of $X.$ This surface can also be seen as the union of all geodesics starting at a assymptotic point of $G(q)$ and ending at $p.$

    If $\Gamma=G(q_1) \cup G(q_2),$ the result is follows from Theorem \ref{teo-exis-par}.
\end{prova}

Let us contextualize and observe that we have proved Theorem \ref{main} in the case where $G$ is parabolic. In this case, we assume that $G$ is the group associated with the parabolic translations $(x,y,z) \mapsto (x, y+k, z)$, $k \in \R$. Then the open connected sets $\Omega_k$, $k>0$, in the theorem are
$$
\Omega_1=\{(x,y,z) \in \overline{\h^3}\mid |x|>|x_1(z)|\},
$$
where $x_1$ is given by \eqref{eq-x1}, and $\Omega_k=\frac{1}{\sqrt{k}}\Omega_1$. The curve $C$ at $\partial_{\infty}\h^3$ from item (iii) of the theorem is $\{x=0\}$.

To conclude our comments on the asymptotic Plateau problem, we can also relate our surfaces to the asymptotic Dirichlet problem (ADP). Let us recall two interesting results about parabolic and hyperbolic Killing graphs, respectively.

\begin{teo}\label{teo-Mi}[Theorem 6 of \cite{Miriam}, adapted for $n=2$]
Let $p$ be a point in the asymptotic boundary of $\h^3$
and $\Gamma \subset \partial_{\infty}\h^3$ be a compact embedded curve passing through $p$. Assume that there are two circles $E_1$ and $E_2$ of $\partial_{\infty}\h^3$ which
are tangent at $p \in E_1\cap E_2$ and such that $\Gamma$ is contained between $E_1$ and $E_2$,
and moreover that any circle of $\partial_{\infty}\h^3$ passing through $p$ orthogonal to $E_1$
intersects $\Gamma$ at exactly one point.
Then, given $|H| < 1$, there exists a unique properly embedded complete
$C^{\infty}$ surface $\Sigma$ of $\h^3$ with CMC $H$ such that $\partial_{\infty}\Sigma =\Gamma$ and $\overline{\Sigma}=\Sigma \cup \Gamma$
is a compact embedded topological surface of $\overline{\h^3}$.
Moreover, $\Sigma$ is a parabolic graph, that is, there exists a totally geodesic
surface $\h^2$ of $\h^3$, a Killing field $Y$ whose orbits in $\h^3$ are horocycles
orthogonal to $\h^2$ and whose orbits in $\partial_{\infty}\h^3$ are the circles orthogonal to $E_1$ at $p$, and a
function $u \in C^{\infty}(\h^3)\cap C^0 (\h^3\setminus\{p\})$ such that $\Sigma$ is the $Y$-graph of $u$.
\end{teo}

The next result is stated in the half-space model of $\h^3$.

\begin{teo}\label{teo-Sp}[Theorem 4.8 of \cite{Spruck}, also adapted for $n=2$]
Suppose $\Gamma$ is the boundary of a star-shaped $C^0$ domain in $\mathbb{R}^2=\partial_\infty\h^3$
and let $|H| < 1$. Then there exists a unique hypersurface $\Sigma$ of constant
mean curvature $H$ in $\h^3$ with asymptotic boundary $\Gamma$. Moreover, $\Sigma$ may
be represented as the radial graph over the upper hemisphere $\mathbb{S}^2_+\subset \mathbb{R}^3$ of a
function in $C^{\infty}(\mathbb{S}^2_+)\cap C^0 \left(\overline{\mathbb{S}^2_+}\right)$.
\end{teo}

We observe that our existence result (Theorem \ref{teo-exis-par}) cannot be obtained as a consequence of any of the above results. First $C_1 \cup C_2 \subset \partial_{\infty}\h^3$ is not a parabolic graph for any parabolic Killing field in $\h^3$. One may see this by fixing the model of $\h^3$ and considering, without loss of generality, the Killing field $X=\frac{\partial}{\partial y}$. Any parabolic Killing field has as orbits in $\partial_{\infty}\h^3$ either a family of parallel lines or a family of tangent circles meeting at a point $p$. In both cases, $C_1 \cup C_2$, which is the union of two orbits, does not intersect every line $\{x=c\}_{c\in \mathbb{R}}$ transversally at a single point.

For the hyperbolic Dirichlet problem considered by Guan and Spruck, which is already stated in the upper half-space model, we observe that the region between two parallel lines $C_1$ and $C_2$ is an unbounded star-shaped domain; hence the theorem does not apply because of the unboundedness. One could see $C_1 \cup C_2$ bounding a bounded region by applying an isometry of $\h^3$, but then it would be the region between two tangent circles, which is not star-shaped. Nevertheless, our surface $\Sigma$ described in the proof of Theorem \ref{teo-exis-par} is a hyperbolic graph defined in $\overline{\mathbb{S}^2_+}\setminus\{p,q\}$ for $p$ and $q$ such that the line $pq$ is parallel to $C_1$, it is the graph of an unbounded function.

We have thus exhibited an unbounded boundary datum at the asymptotic boundary of $\h^2$ (in this case seen as the upper hemisphere of $\mathbb{S}^2$) for which the asymptotic Dirichlet problem has a solution, which we presented explicitly. 

%
%

\subsection{Hyperbolic Killing fields}

In this subsection we exhibit a foliation of $\h^3$ by hyperbolic minimal surfaces.

\begin{defi}
    We say that a Killing field in $\h^3$ is hyperbolic if the orbits of $X$ are hypercircles, that is, curves equidistant from a geodesic.
\end{defi}

Let $f:(-\pi,\pi)\times\R\times (0,+\infty) \rightarrow \h^3$ be a parametrization of $\h^3$ given by
\begin{equation} \label{f.hyp}
f(\phi,s,\tau)=e^s(\cos{\phi}\tanh{\tau},\sin{\phi}\tanh{\tau},\sech \tau).
\end{equation}
To simplify, we identify points in $\h^3$ with their coordinates.

As in the parabolic case, given a hyperbolic Killing field $X$, we choose a totally geodesic $\h^2\subset\h^3$ which is orthogonal to $X$. We consider this $\h^2$ as the subset
\[
\h^2:=\{(\phi,s,\tau)\in \h^3 \mid s=0\}
\]
and the hyperbolic Killing field as $X=a\,\partial/\partial s$, with $a\neq 0$. Since we are interested in invariant surfaces, we may assume $a=1$. Therefore we can use this model to compute the metric $g^*$.

Let $G$ be the $1$-parameter subgroup of $Iso(\h^3)$ associated with $X$. Since $X$ is orthogonal to $\h^2$, we can identify $\h^2$ with $\h^3/G$; we just need to define the projection
\begin{eqnarray*}
    \pi: \h^3 & \rightarrow & \h^2 \\
    (\phi,s,\tau) & \mapsto & (\phi,0,\tau).
\end{eqnarray*}
Again, since $X$ is orthogonal to $\h^2$, the metric $g$ that makes $\pi$ a Riemannian submersion is precisely the metric $\overline{g}$ restricted to $\h^2$.

Now we determine the metric $g^*$ on $\h^3/G$ from the hypotheses of the Reduction Theorem.

Note that $\displaystyle |X|^2=\cosh^2{\tau}$. Then the metric $g^*$ is given by
\[
ds^2=\cosh^2 \tau \sinh^2{\tau}\,d\phi^2+\cosh^2 \tau\, d\tau^2.
\]
Thus we can compute the sectional curvature $K$ of $(\h^3/G,g^*)$:
\[
K(\tau)=-2\sech^2{\tau}-\sech^4{\tau}.
\]
Therefore $(\h^3/G,g^*)$ is a Hadamard rotationally symmetric manifold and we obtain the following result as an application of Corollary \ref{coro}.

\begin{teo}\label{sup.hyp}
For each hyperbolic Killing field in $\h^3$ there exists a foliation of $\h^3$ by complete properly embedded simply connected hyperbolic minimal surfaces.
\end{teo}

As in the parabolic case, we can explicitly obtain a family of hyperbolic invariant surfaces.

We want to find a curve $\gamma:\mathbb{R}\rightarrow\mathbb{H}^2$, which in coordinates is given by
\[
\gamma(t)=\big(\cos\phi(t)\tanh{\tau(t)},\sin\phi(t)\tanh{\tau(t)},\sech{\tau(t)}\big),
\]
such that the flow of $X$ through each of its points generates a minimal surface $\Sigma\subset\mathbb{H}^3$. By writing equation (\ref{eq.sup.inv.}) in these coordinates, we obtain the following system:

\begin{equation}\label{sis.hiperb}
\begin{cases}
(\tau''\phi'-\phi''\tau')\tanh \tau-\phi'(\tau'^2+\tanh^2\tau+1) = 0 \\
\tau'^2+\phi'^2\sinh^2\tau=1
\end{cases}
\end{equation}

Again, it is possible to solve the system (\ref{sis.hiperb}). The solutions are

\begin{eqnarray*}
   \tau_k(t) & = & \tanh^{-1}\left[\left(\frac{\sqrt{4k^2+1}\cosh(2t)-1}{\sqrt{4k^2+1}\cosh(2t)+1}\right)^{1/2}\right], \\
    \\
    \phi_k(t) & = & \int_0^t\frac{2k\sqrt{2}}{\big(\sqrt{4k^2+1}\cosh(2s)-1\big)\sqrt{\sqrt{4k^2+1}\cosh(2s)+1}}\,ds.
\end{eqnarray*}

A complete version of this approach can be found in \cite{eu}.

Hence, for each $k>0$ there exists a hyperbolic simply connected embedded minimal surface
\[
\Sigma_k=\{\big(e^s\cos{\phi_k(t)}\tanh{\tau_k(t)},e^s\sin{\phi_k(t)}\tanh{\tau_k(t)},e^s\sech{\tau_k(t)}\big)\mid t\in \R, s\in \R\}.
\]

\begin{obs}
The initial condition is chosen so that the curve $\gamma$ is orthogonal to the geodesic $\alpha(\tau) := f(0,0,\tau)$ at $t=0$, that is, at $\alpha(\tau_k(0))$.
\end{obs}

\begin{obs}
    We believe that this family of minimal surfaces was already studied by do Carmo and Dajczer and appears in Theorem 5.5 of \cite{Manfredo}.
\end{obs}

Since we have the parametrization, we can use the software GeoGebra to plot the surface:

\begin{figure}[H]
    \centering
    \caption{Hyperbolic minimal surface $\Sigma$}
    \vspace{0.5cm}\includegraphics[scale=0.45]{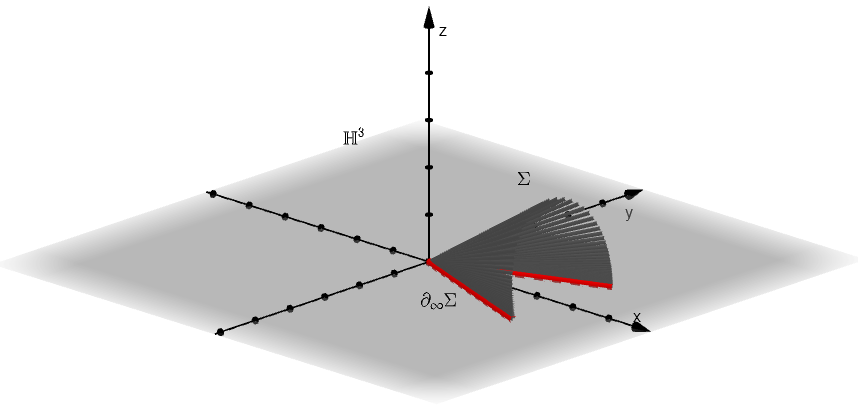}
\end{figure}

The generating curves $\gamma_k$ can also be parametrized with $\phi_k$ as a bigraph of a function of $\tau$:

$$
\phi_k(\tau)=\pm k\int_{\tau_{k,0}}^\tau 
\frac{1}{\sinh(x)\sqrt{\sinh^2(x)\cosh^2(x)-k^2}}\,dx,
$$
where $\tau_{k,0}$ is given by $$\sinh^2(\tau_{k,0})=\frac{\sqrt{4k^2+1}-1}{2} \text{ and }\tau_{k,0}>0.$$

The asymptotic boundary of $\Sigma_k$ is clearly the orbit of $\gamma_k(-\infty)$ and the orbit of $\gamma_k(+\infty)$, which, if one uses the polar coordinates $(e^s,\phi)$ induced by $f$ on $\partial_{\infty}\h^3$, correspond to the rays $\phi=\pm \phi_k(\infty)$.

Up to now we have carried out every explicit computation in half of $\h^3$, considering only the region $x>0$. The family of surfaces $\Sigma_k$ was obtained following the proof of the Foliation Theorem (Theorem \ref{teo-fol}) and therefore they foliate $\h^3\cap \{x>0\}$. In order to foliate the whole hyperbolic space, we must include the leaf corresponding to $k=0$, which we define as the totally geodesic minimal surface $\Sigma_0=\{x=0\}$, and the leaves $\Sigma_{-k}=\rho(\Sigma_k)$, where $\rho$ is the reflection across $\Sigma_0$.

We can now observe that we have proved Theorem \ref{main} for $G$ a hyperbolic subgroup of the isometry group of $\h^3$. Let us define
$$
\Omega_s=\bigcup_{k<s} \Sigma_k \cup D,
$$
for $D=\{(x,0,0)\in \partial_\infty \h^3\mid x<0\}$. These domains $\Omega_s$ are as in the statement of the theorem if, in item (iii), one considers $C=\{y=0\}\subset \partial_\infty \h^3$.

One can also look at these surfaces through the perspective of the asymptotic Plateau problem: 

\begin{teo}\label{teo-exis-hip}
Let $X$ be a hyperbolic Killing field in $\h^3$ and let $p,q\in \partial_{\infty}\h^3$ be the singularities of $X$. Let $q_1$ and $q_2$ be two points in $\partial_{\infty}\h^3\backslash\{p,q\}$ with distinct orbits $C_1$ and $C_2$ associated with $X$. Then there exists a unique complete hyperbolic minimal surface $\Sigma$ with asymptotic boundary $C_1 \cup C_2$.
\end{teo}

\begin{prova}
Let $q_1$ and $q_2$ be the points in $\partial_\infty \mathbb{H}^3 / G$ corresponding to the orbits $C_1$ and $C_2$, respectively. If we show that there exists a geodesic in $(\mathbb{H}^3 / G, g^*)$ connecting $q_1$ to $q_2$, then Theorem~\ref{red} yields the desired result. 

To prove the existence of such a geodesic, observe that $(\mathbb{H}^3 / G, g^*)$ is rotationally symmetric. Let $o$ is the origin of $(\mathbb{H}^3 / G, g^*)$ and let $m\in \partial_\infty \mathbb{H}^3 / G$  be such that the reflection of $(\mathbb{H}^3 / G, g^*)$ across the geodesic $\alpha$ starting from $o$ and asymptotic to $m$ maps $q_1$ to $q_2$. Since $(\mathbb{H}^3 / G, g^*)$ is rotationally symmetric, this reflection is an isometry.

Since $(\mathbb{H}^3 / G, g^*)$ is a Hadamard manifold, Corollary~\ref{coro} implies that it admits a foliation by geodesics orthogonal to $\alpha$. Furthermore, since $(\h^3/G,g^*)$ is rotationally symmetric, a leaf $\gamma_k$ of this foliation must connect $q_1$ and $q_2$; see Figure \ref{provaa}.

\begin{figure}[H]
\centering
\caption{Illustration of Theorem \ref{teo-exis-hip}}
\label{provaa}
\vspace{0.5cm}
\includegraphics[scale=0.25]{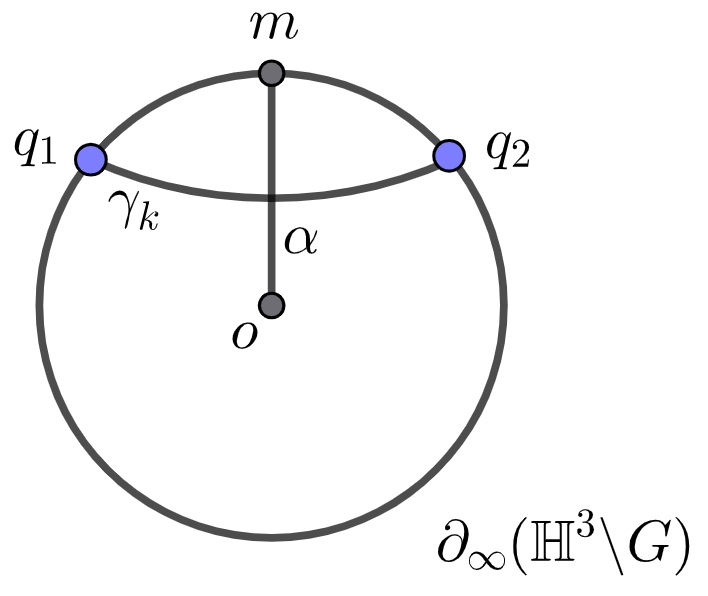}
\end{figure}\textcolor{white}{.}
\end{prova}

As in Corollary \ref{cor-11par}, Theorem \ref{exis} for $G$ is a hyperboilic subgroup of Iso($\h^3$) is as consequence of \ref{teo-exis-hip}. In this case, a single orbit is not a closed curve in $\partial_\infty\h^3,$ so the Theorem reduces to considering the union of two orbits.

We conclude this section by relating the hyperbolic family of surfaces to the asymptotic Dirichlet problem described in Theorems \ref{teo-Mi} and \ref{teo-Sp}. First we observe that, once again, the asymptotic boundary of the hyperbolic minimal surfaces discussed above is not a parabolic graph in $\partial_\infty \h^3$, except in the trivial case of a totally geodesic surface.

Nevertheless, the hyperbolic minimal surfaces are indeed hyperbolic Killing graphs of a bounded function defined on a totally geodesic $\h^2 \subset \h^3$. In order to see this we must deal with two hyperbolic Killing fields: the first one is associated with the invariance of the surface and the second determines the direction of the Killing graph considered. We consider $Y=\partial/\partial s$ associated with the Killing graph. Therefore we must apply an isometry to our surfaces so that they become invariant under the action of another hyperbolic Killing flow.

According to Section 1.1 of \cite{cambri}, the Möbius transformation
$$
\varphi(z)=\frac{z-1}{z+1}
$$
defined for $z\in \mathbb{C}$, with $\mathbb{C}$ identified with $\partial_\infty \h^3$, can be extended to $\h^3$ as an isometry by
$$
\Psi(z,t)=\left(\frac{-2\bar{z}-2}{|z+1|^2+t^2}+1, \frac{2}{|z+1|^2+t^2}\right),\qquad (z,t)\in \mathbb{C}\times \R^+=\h^3.
$$

Here one must consider coordinates given by \eqref{eq-H3ret} in $\h^3$.

The Möbius transformation $\varphi$ takes the unbounded domain 
$$
\Omega_{\alpha}=\{(x,y)\in \R^2 \mid x>0 \text{ and } |y|<x \tan \alpha\}
$$
into the domain bounded by two arcs of circles $C^+$ and $C^-$ connecting $(-1,0)$ to $(1,0)$ and contained in the circles of center $(0,-\cot \alpha)$ and radius ${\rm csc\,}\alpha$.

\begin{figure}[H]
\centering
\caption{Illustration of Theorem \ref{teo-exis-hip}}
\label{provaa}
\vspace{0.5cm}
\includegraphics[scale=0.65]{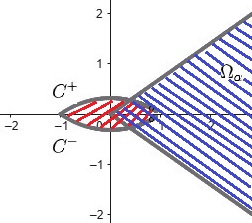}
\end{figure}

Since, according to Theorem \ref{teo-Sp}, there is a unique minimal surface bounded by $C^+$ and $C^-$, the surface $\Psi(\Sigma_{k_0})$, for some $k_0$, is the radial graph given by Theorem \ref{teo-Sp}. Consequently the invariant minimal surfaces can be seen as a particular explicit solution of the hyperbolic asymptotic Dirichlet problem (ADP).

Moreover, we have also proved that if the asymptotic boundary of a star-shaped domain is invariant under the flow of a hyperbolic Killing field, then the solution of the ADP is also invariant under this flow. The relation $\Phi(k_0)=\alpha$, where $\Phi(k)$ denotes the angular half-width of $\Sigma_k$, is given by
\begin{eqnarray}
\Phi(k)&=&\displaystyle{\int_0^{\infty}\frac{2k\sqrt{2}}{\big(\sqrt{4k^2+1}\cosh(2s)-1\big)\sqrt{\sqrt{4k^2+1}\cosh(2s)+1}}\,ds}\nonumber\\
~& &\nonumber\\
&=&\displaystyle{\int_{u_0(k)}^{\infty} \frac{k}{(u^2-1)\sqrt{u^4-u^2-k^2}}\,du}
\nonumber        
\end{eqnarray}
for
$$
u_0(k)=\left(\frac{\sqrt{4k^2+1}+1}{2}\right)^{1/2},
$$
which provides the way to determine the corresponding surface $\Psi(\Sigma_{k_0})$.

\subsection{Helicoidal Killing fields}

In this subsection we prove that there exists a foliation of $\h^3$ by helicoidal minimal surfaces. First we define a rotational Killing field.

\begin{defi}
    We say that a Killing field $X$ in $\h^3$ is rotational if the orbits of $X$ are geodesic circles whose centers lie on a fixed geodesic.
\end{defi}

Now we define a helicoidal Killing field.

\begin{defi}
    We say that a Killing field $X$ in $\h^3$ is helicoidal if $X=Y+Z$, where $Y$ is hyperbolic, $Z$ is rotational and $[Y,Z]=0$.
\end{defi}

The condition on the Lie brackets ensures that the rotation occurs around the geodesic invariant under the hyperbolic Killing field. Thus the orbits of a helicoidal Killing field are a kind of helix, as in $\R^3$, whose axis is the geodesic fixed by the rotation.

Let $f$ from \eqref{f.hyp} be a parametrization of $\h^3$. Let $X$ be a helicoidal Killing field which we assume to be
\[
X=a\frac{\partial f}{\partial s}+b\frac{\partial f}{\partial\phi},
\]
where $a$ and $b$ are real parameters representing the pitch of the helix determined by the orbits of $X$, with $a\neq 0$. Since we want invariant surfaces, we may assume $a=1$. Let $\h^2$ be given by
\[
\h^2:=\{(\phi,s,\tau)\in\h^3 \mid s=0\},
\]
which is a totally geodesic surface transversal to $X$. We use this model to compute the metric $g^*$.

Again, let $G\subset Iso(\h^3)$ be the one-parameter subgroup associated with $X$. This case is different from the previous ones, since the Killing field $X$ is not orthogonal to $\h^2$. Thus we first need to compute the metric $g$ for which $\pi:\h^3\rightarrow (\h^3/G,g)$ is a Riemannian submersion. Since the vector field $X$ is not tangent to $\h^2$, we can equip $\h^2$ with a metric $\hat{g}$ such that $(\h^2,\hat{g})$ is isometric to $(\h^3/G,g)$. We now determine this metric $\hat{g}$.

Given $p\in\h^3$, let $O_p:T_p\h^3\rightarrow T_pG(p)$ be the orthogonal projection onto the orbit of $p$. Then we can determine the metric $\hat{g}$ by
\begin{equation}\label{proj}
    \hat{g}(Y,Z)(p) = \overline{g}(Y,Z)(p)+\overline{g}\big(O_p(Y),O_p(Z)\big),
\end{equation}
where $Y,Z\in T_p\h^2$ and $\overline{g}$ denotes the canonical metric on $\h^3$. Note that if the orbit through $p$ is orthogonal to $\h^2$ for all $p$ in $\h^2$, equation \eqref{proj} shows that the metric $\hat{g}$ coincides with the metric of $\h^3$ restricted to $\h^2$, which is the canonical metric of $\h^2$, as in the two previous cases. 

Now, substituting equation \eqref{proj} into the model for $\h^3$ discussed above, the metric $\hat{g}$ is given by
\[
ds^2=
\Bigg(\frac{\cosh^2{\tau}+2b^2\sinh^2{\tau}}{\cosh^2{\tau}+b^2\sinh^2{\tau}}\Bigg)\sinh^2{\tau}\,d\phi^2+d\tau^2,
\]
since $|X|^2=\cosh^2{\tau}+b^2\sinh^2{\tau}$. Hence the metric $g^*$ is given by
\[
ds^2=(\cosh^2{\tau}+2b^2\sinh^2{\tau})\sinh^2{\tau}\,d\phi^2+(\cosh^2{\tau}+b^2\sinh^2{\tau})\,d\tau^2.
\]

Then the sectional curvature $K$ of $(\h^3/G,g^*)$ is

\begin{eqnarray}\label{curv.hel}
    K(\tau) & = & -\frac{1}{|X|^4\big(|X|^2+b^2(\sinh{\tau})^2\big)^{2}}\bigg(\cosh^6{\tau}\big(10b^2+16b^4+8b^6+2\big) + \nonumber \\ \nonumber
    \\ \nonumber
    & + & \cosh^4{\tau}\big(-5b^2-26b^4-24b^6+1\big) + \cosh^2{\tau}\big(12b^4+24b^6\big)+ \\ \nonumber
    \\ 
    & + & \big(-2b^4-8b^6\big)\bigg).
\end{eqnarray}

Let $g_{b}: \R \rightarrow\R$ be the polynomial function defined by
\begin{eqnarray*}
    g_{b}(x) & = & x^6\big(10b^2+16b^4+8b^6+2\big)+x^4\big(-5b^2-26b^4-24b^6+1\big) + \\
    \\
    & + &x^2\big(12b^4+24b^6\big)+\big(-2b^4-8b^6\big).
\end{eqnarray*}

If $g_{b}$ is positive on $[1,\infty)$, then $K$ is negative. Note that
\[
g_{b}(1)=3+5b^2+12b^4,
\]
which is positive. Moreover, one can prove that $g_{b}'(x)>0$ for all $x>1$. Hence $g_{b}$ is positive. Therefore $(\h^3/G,g^*)$ is a Hadamard manifold and we obtain the following result as an application of Corollary \ref{coro}.

\begin{teo}\label{teohel}
  For each helicoidal Killing field in $\h^3$ there exists a foliation of $\h^3$ by complete properly embedded simply connected helicoidal minimal surfaces. 
\end{teo}

Since we do not have a parametrization for the helicoidal surface, we can only present a sketch of the surface:

\begin{figure}[H]
    \centering
    \caption{Helicoidal minimal surface $\Sigma$}
    \vspace{0.5cm}\includegraphics[scale=0.45]{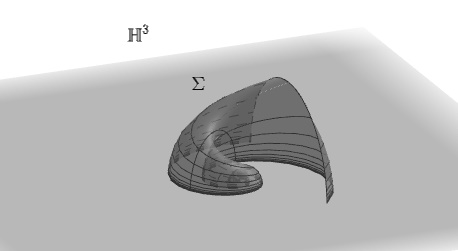}
\end{figure}

As in the hyperbolic and parabolic cases, we now consider these surfaces in the context of the asymptotic Plateau problem and obtain the following result.

\begin{teo}\label{teo-exis-heli}
Let $X$ be a helicoidal Killing field in $\h^3$ and let $p,q\in \partial_{\infty}\h^3$ be the singularities of $X$. Let $q_1$ and $q_2$ be two points in $\partial_{\infty}\h^3\backslash\{p,q\}$ with distinct orbits $C_1$ and $C_2$ associated with $X$. Then there exists a unique complete helicoidal minimal surface $\Sigma$ with asymptotic boundary $C_1 \cup C_2$.
\end{teo}

\begin{prova}
See the proof of Theorem \ref{teo-exis-hip}.
\end{prova}

Now Theorem \ref{main} for the case where $X$ is a helicoidal Killing field follows from Theorem \ref{teohel} and Theorem \ref{teo-exis-heli}.

Observe that Theorem \ref{exis} holds if $G$ is a helicoidal group of isometries. This is as a consequence of the last result following the same ideas of Corollary \ref{cor-11par}. Therefore, in order to conclude the proof of Theorem \ref{exis}, one needs only to consider the $G$ an elliptic group. In this case, Remark \ref{obss} shows the family of hyperbolic planes in $\h^3$ has the property to foliates $\overline{\h}^3$. Hence, one can note that these surfaces solve Theorem \ref{exis} for the case where $G$ is an elliptic Killing field.

We present some examples of curves $\Gamma$ that can be minimally filled according to Theorem \ref{exis}. We use the upper half-space model and the sphere $\mathbb{S}^2$ to represent $\partial_\infty\h^3$.

\begin{figure}[H]
    \centering
    \caption{$\Omega$ when $X$ is rotational}
    \vspace{0.5cm}\includegraphics[scale=0.40]{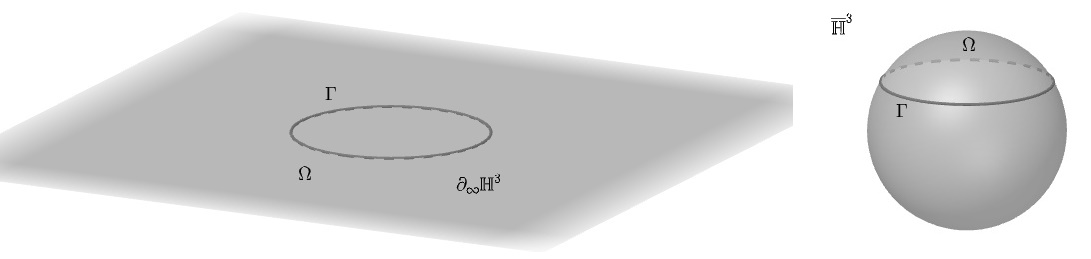}
\end{figure}

\begin{figure}[H]
    \centering
    \caption{$\Omega$ when $X$ is parabolic}
    \vspace{0.5cm}\includegraphics[scale=0.40]{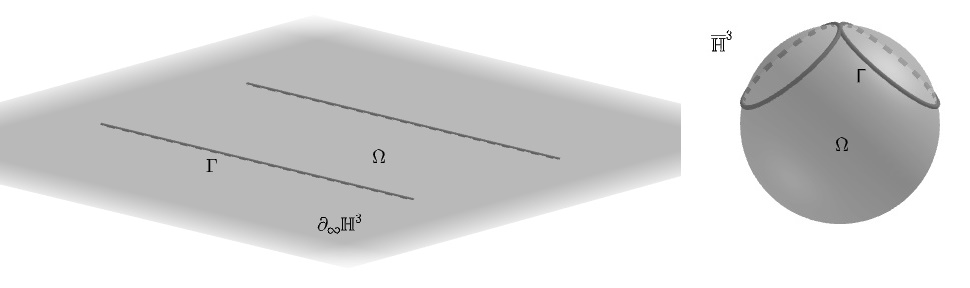}
\end{figure}

\begin{figure}[H]
    \centering
    \caption{$\Omega$ when $X$ is hyperbolic}
    \vspace{0.5cm}\includegraphics[scale=0.40]{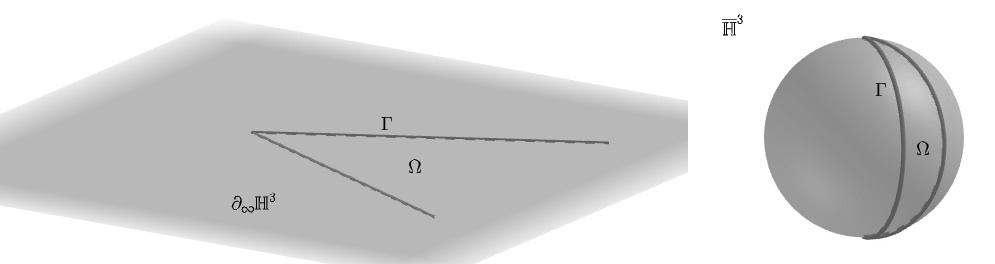}
\end{figure}

\begin{figure}[H]
    \centering
    \caption{$\Omega_s$ when $X$ is helicoidal}
    \vspace{0.5cm}\includegraphics[scale=0.40]{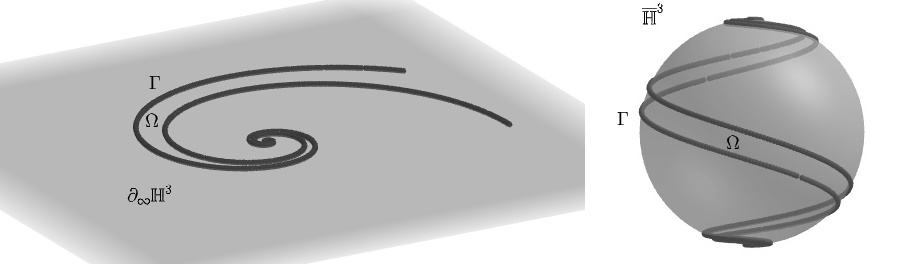}
\end{figure}



\end{document}